\documentclass{article}

\PassOptionsToPackage{utf8}{inputenc}
\usepackage{inputenc}

\usepackage[export]{adjustbox}
\usepackage{amsfonts}			% math
\usepackage{amsmath}			% math
\usepackage{amssymb}			% math
\usepackage{amsthm}				% math: define theorems etc
\usepackage{aligned-overset} % align oversets
\usepackage{arxiv}        % article template
\usepackage{bbm}          % bold 1
\usepackage[backend=biber,sorting=nty,isbn=false,doi=false,maxnames=999,giveninits=true,sortcites=true,url=false]{biblatex}       % bibtex
\usepackage{blindtext} 			 % lorem ipsum
\usepackage[labelsep=period,justification=justified]{caption}
\usepackage[dvipsnames]{xcolor}	% namen
\usepackage{enumitem,xparse}   % indented claim-inside-proof thing
\usepackage{csquotes}          % quotation marks in hübsch
\usepackage{float}				% figures
\usepackage{geometry}
\usepackage{graphicx}			% figures
\usepackage{mathrsfs}			% math
\usepackage{mathtools}			% fixes amsmath
\usepackage{parskip}      % no indentation in paragraphs, instead spacing
\usepackage{pgfplots}
\pgfplotsset{compat=1.18}
\usepackage{subcaption}			% subfigures
\usepackage{thmtools}       % theorem thingies
\usepackage{tikz}         % tikz
\usetikzlibrary{arrows}
\usetikzlibrary{arrows.meta}  % pfeilspitzen       
\usetikzlibrary{arrows.meta}  %pfeilspitzen       
\usetikzlibrary{backgrounds} %versch layer
\usetikzlibrary{calc} %coordinate calculation
\usetikzlibrary{decorations.pathreplacing}  %lange klammer
\usetikzlibrary{fit} %kreis um nodes
\usetikzlibrary{patterns}
\usetikzlibrary{positioning}
\usetikzlibrary{tikzmark}
\usepackage{url}				% urls in bibtex hopefully
\usepackage{xparse}     % extremely sophisticated newcommands

\usepackage[english]{babel} % babel system, adjust the language of the content

\usepackage[colorlinks = true,linkcolor = blue,urlcolor  = blue,citecolor = blue,anchorcolor = blue]{hyperref}   % clicky references
\usepackage[capitalise]{cleveref}  %named references

\crefname{lemma}{Lemma}{Lemmata}
\crefname{subsection}{Subsection}{Subsections}

\let\olditemize=\itemize
\let\endolditemize=\enditemize
\renewenvironment{itemize}{\olditemize \itemsep0em}{\endolditemize}

\hypersetup{pdftitle={A Random Walk Approach to Broadcasting on Random Recursive Trees}}

\newtheoremstyle{aussagen}
{}                % space above
{}                % space below
{\slshape}        % font of theorem body
{}                % indent of theorem
{\bfseries}       % font of theorem head
{}                % punctuation between head and body
{0.5em}           % space after Theorem head
{\thmname{#1}\thmnumber{ #2}. \thmnote{(#3)}}  % Specification of theorem head

\newtheorem{theorem}{Theorem}[section]
\newtheorem{lemma}[theorem]{Lemma}

\newtheorem{definition}[theorem]{Definition}
\newtheorem{remark}[theorem]{Remark}

\title{A Random Walk Approach to Broadcasting \\on Random Recursive Trees}
\newcommand{\shorttitle}{A Random Walk Approach to Broadcasting on Random Recursive Trees}
\date{\today}

\author{\hspace{1mm}Ernst Althaus \\
	Institute of Computer Science\\
	Johannes Gutenberg University\\
	Mainz, Germany \\
	\And
	\hspace{1mm}Lisa Hartung \\
Institute of Mathematics\\
  Johannes Gutenberg University\\
  Mainz, Germany \\
	\And
	\hspace{1mm}Rebecca Steiner \\
  Institute of Mathematics\\
  Johannes Gutenberg University\\
  Mainz, Germany \\
}

\definecolor{darkblue}{RGB}{40,40,120}
\definecolor{darkred}{RGB}{220,20,20}
\definecolor{jgurot}{RGB}{193,0,42}
\definecolor{jgublue}{RGB}{0,138,193}
\colorlet{red}{jgurot}
\colorlet{blue}{jgublue}

\let\pgfimageWithoutPath\pgfimage
\renewcommand{\pgfimage}[2][]{\pgfimageWithoutPath[#1]{bildchen/#2}}

\def\todos{1}
\if\todos1
  \newcommand{\todo}[1]{\textcolor{Maroon}{{\textbf [TODO: #1]}}}
\else
  \newcommand{\todo}[1]{}
\fi

\def\comments{1}
\if\comments1
  \newcommand{\cmnt}[1]{\textcolor{Green}{{[Comment: #1]}}}
\else
  \newcommand{\cmnt}[1]{}
\fi

\def\mycomments{1}
\if\mycomments1
  \newcommand{\mycmnt}[1]{\textcolor{TealBlue}{{[My Comment: #1]}}}
\else
  \newcommand{\mycmnt}[1]{}
\fi

\DeclarePairedDelimiterXPP\Prob[1]{\mathbb{P}}{(}{)}{}{
  \providecommand\given{\nonscript\:\delimsize\vert\nonscript\:\mathopen{}}
  #1}

\DeclarePairedDelimiterXPP\Exp[1]{\mathbb{E}}{[}{]}{}{
  \providecommand\given{\nonscript\:\delimsize\vert\nonscript\:\mathopen{}}
  #1}

\DeclarePairedDelimiterXPP\Var[1]{\mathrm{Var}}{[}{]}{}{
  \providecommand\given{\nonscript\:\delimsize\vert\nonscript\:\mathopen{}}
  #1}

\DeclarePairedDelimiterXPP\Cov[1]{\mathrm{Cov}}{[}{]}{}{
  \providecommand\given{\nonscript\:\delimsize\vert\nonscript\:\mathopen{}}
  #1}

\let\deg\relax
\DeclarePairedDelimiterXPP\deg[2]{\mathrm{deg}_{#1}}{(}{)}{}{#2}
\DeclarePairedDelimiterXPP\degpl[2]{\mathrm{deg}^+_{#1}}{(}{)}{}{#2}

\DeclarePairedDelimiterX\skp[2]{\langle}{\rangle}{#1, #2}

\DeclarePairedDelimiter\abs{\lvert}{\rvert}

\DeclarePairedDelimiter\ceil{\lceil}{\rceil}

\DeclarePairedDelimiterX\Set[1]{\{}{\}}{
  \providecommand\given{\nonscript\:\delimsize\vert\nonscript\:\mathopen{}}
  #1}

\let\originalleft\left
  \let\originalright\right
\renewcommand{\left}{\mathopen{}\mathclose\bgroup\originalleft}
  \renewcommand{\right}{\aftergroup\egroup\originalright}

\newcommand{\cF}{\mathcal{F}}
\newcommand{\cA}{\mathcal{A}}

\newcommand{\cT}{\mathcal{T}}

\newcommand{\N}{\mathbb{N}}
\newcommand{\E}{\mathbb{E}}

\newcommand{\1}{\mathbbm{1}}

\newcommand{\bmaj}{b_{\textup{maj}}}
\newcommand{\Rmaj}{R_{\textup{maj}}}

\newcommand{\taubad}{\tau_{\text{\upshape low}}}
\newcommand{\taugood}{\tau_{\text{\upshape high}}}

\newcommand{\Zan}{Z_\alpha(n)}
\newcommand{\ZaN}{Z_\alpha(N)}

\newcommand{\rr}{(r,r)}
\newcommand{\rb}{(r,b)}
\newcommand{\br}{(b,r)}
\newcommand{\bb}{(b,b)}
\newcommand{\rw}{(r,\_)}
\renewcommand{\wr}{(\_,r)}
\newcommand{\bw}{(b,\_)}
\newcommand{\wb}{(\_,b)}

\newcommand{\srrN}{\#_{(r,r)([N])}}
\newcommand{\srbN}{\#_{(r,b)([N])}}
\newcommand{\sbrN}{\#_{(b,r)([N])}}
\newcommand{\sbbN}{\#_{(b,b)([N])}}
\newcommand{\srwN}{\#_{(r,\_)([N])}}
\newcommand{\swrN}{\#_{(\_,r)([N])}}
\newcommand{\sbwN}{\#_{(b,\_)([N])}}
\newcommand{\swbN}{\#_{(\_,b)([N])}}

\DeclareMathOperator{\sgn}{sgn}
\DeclareMathOperator{\Vari}{Var}

\newcommand{\conv}[1]{\overset{\text{#1}}{\longrightarrow}}
\newcommand{\colvec}[2]{(#1,#2)}
\newcommand{\bigcolvec}[2]{\big(#1,#2\big)}

\makeatletter
\renewcommand{\nsim}{\mathrel{\mathpalette\n@sim\relax}}
\newcommand{\n@sim}[2]{%
  \ooalign{%
    $\m@th#1\sim$\cr
    \hidewidth$\m@th#1\rotatebox[origin=c]{60}{$#1-$}$\hidewidth\cr
  }%
}
\makeatother

\makeatletter
\newcommand*\rel@kern[1]{\kern#1\dimexpr\macc@kerna}
\newcommand*\widebar[1]{%
  \begingroup
  \def\mathaccent##1##2{%
    \rel@kern{2.2}%
    \overline{\rel@kern{-1.8}\macc@nucleus\rel@kern{0.1}}%
    \rel@kern{-0.4}%
  }%
  \macc@depth\@ne
  \let\math@bgroup\@empty \let\math@egroup\macc@set@skewchar
  \mathsurround\z@ \frozen@everymath{\mathgroup\macc@group\relax}%
  \macc@set@skewchar\relax
  \let\mathaccentV\macc@nested@a
  \macc@nested@a\relax111{#1}%
  \endgroup
}
\makeatother

\newcommand\addtag[1][\theequation]{\refstepcounter{equation}\tag{#1}}

\begin{document}

\maketitle

\begin{abstract}
 In the broadcasting problem on trees, a $\{-1,1\}$-message originating in an
  unknown node is passed along the tree with a certain error probability $q$. The goal is to estimate the original message without knowing the order in which the nodes were informed. We show a connection to random walks with memory effects and use this to develop a novel approach to analyse the majority estimator on random recursive trees. With this powerful approach, we study the entire group of very simple increasing trees as well as shape exchangeable trees together. This also extends Addario-Berry et al.~(2022) %~\cite{ADLV22} 
  who investigated this estimator for uniform and linear preferential attachment random recursive trees.
\end{abstract}


%\ams{60C05}{82C41; 05C80} 

\section{Introduction}
Incrementally growing random trees and networks are important building blocks in understanding the formation of networks and their structural properties. Analyzing how potentially false information may spread in such a network is a complicated task with many possible modelling approaches. In this
article we study the broadcasting process for two classes of growing random trees. A growing random tree is a sequence of trees $(T_n)_n$ with $T_1$ containing the isolated vertex $1$ and $T_{n+1}$ constructed out of $T_n$ by attaching the vertex $n+1$ to $T_n$ via one new edge. In these trees, the paths from vertex $1$, the root, to all leaves have increasing ages, earning them the label of increasing or \emph{recursive} trees. The attachment point of vertex $n+1$ is chosen according to a given attachment distribution depending only on $T_n$: At each time $n$, each vertex is given a weight and the probability that vertex $n+1$ will attach to it is proportional to this weight. The simplest weighting procedure assigns each vertex the same attachment weight, generating a uniform attachment tree. This tree process has, at size $n$, the same distribution over trees of size $n$ as uniformly choosing one among all possible recursive trees of size $n$. However, different weightings lead to other tree distributions~\cite{Drm09}. A natural next step is allowing dependence of these weights on vertex properties such as the number of (outgoing) edges, also known as its (out)degree. This leads to our two classes: Very simple increasing trees~\cite{Drm09} and shape exchangeable trees~\cite{CX21}. Very simple increasing trees are a family of growing random trees where the attachment weights are a linear function on the outdegree of the vertex. This family of trees separates into three sub-families: Uniform attachment, linear preferential attachment and uniform attachment on a $d$-ary tree. The choice of using the outdegree instead of the total degree of the vertex is motivated by an analytic combinatorics approach~\cite{Flaj09,DW19}. Since we will not be using such an approach, considering a similar family of tree models where the vertices are weighted by their entire degree is a sensible extension. This model group is known as shape exchangeable trees, originally introduced by Crane and Xu~\cite{CX21} in their study of root reconstructability. This family also contains uniform attachment, a different linear preferential attachment and uniform attachment on a $d$-\emph{regular} tree. We will see that this small change also causes some differences in the bounds we obtain, though largely the two groups behave the same, as one would intuitively expect.

Now, the broadcasting process on a growing tree can informally be described as follows: Consider a sequence of
trees $(T_n)_{n\in \mathbb{N}}$ as above. At the beginning, the tree consists of only the root vertex which additionally receives one of two available colors. Each timestep, a new vertex attaches itself to the tree and receives the color of its parent with probability $1-q$ and the opposite color with probability $q$, independent of the other vertices and their colors. The key question we study in this paper is the influence of the root color on the appearance of the colored tree when its size grows to infinity. We phrase this question of the local color-passing interactions influencing the global behavior as an estimation problem and investigate the relation between the color majority and the root color. To analyse the evolution of this color majority, we present a novel modelling approach of it as a space- and time-inhomogeneous random walk which is related to both a Pólya urn process~\cite{Frie49} with random replacements and the reinforced Elephant Random Walk~\cite{Lau22}. We describe our approach in detail in~\cref{sec:results}.

Formally, the two classes of random trees we consider are defined as follows. Let $(T_n)_{n \in \N}$ be a sequence of growing trees, with $T_n$ containing $n$ vertices. The vertices of $T_n$ can naturally be labeled by
$\{1,\dots,n\}$ according to their arrival time, making each $T_n$ a random recursive tree~\cite{Drm09} of size $n$. We denote by $\deg{n}{v}$ the degree of vertex $v$ at time $n$ and by $\degpl{n}{v}$ its outgoing degree, where the outgoing degree only counts edges to vertices with a bigger label, also called \emph{children} of $v$. So, for $v = 1$, $\degpl{n}{v} = \deg{n}{v}$, and for all other vertices, $\degpl{n}{v} = \deg{n}{v} -1$.

\begin{definition}[Very Simple Increasing Tree]\label{def:vsinctree}
  A \emph{very simple increasing (v.s.i.) tree} is a random recursive tree $(T_n)_{n \in \N}$ that
  can be grown iteratively with the following attachment probability
  distribution for each new vertex:
  \[\forall v \in \{1,\ldots,n\}: \: \Prob*{n+1 \sim v \given T_n} =
    \frac{\alpha\degpl{n}{v}+1}{\alpha(n-1) + n}, \addtag\]
  where
  \[\alpha \in \Set*{\tfrac{-1}{d} \given d \in \N_{> 1}} \cup
    [0,\infty).\phantom ] \addtag\]
\end{definition}

\begin{definition}[Shape Exchangeable Tree]\label{def:shextree}
  A \emph{shape exchangeable (s.e.) tree} is a random recursive tree $(T_n)_{n \in \N}$ that
  can be grown iteratively with the following attachment probability
  distribution for each new vertex:
  \[\forall v \in \{1,\ldots,n\}: \: \Prob*{n+1 \sim v \given T_n} =
    \frac{\alpha\deg{n}{v}+1}{2\alpha (n-1) + n}, \addtag\]
  where
  \[\alpha \in \Set*{\tfrac{-1}{d} \given d \in \N_{> 2}} \cup
    [0,\infty).\phantom ] \addtag\]
\end{definition}
The main difference between these two models is found in their treatment of the root vertex - the root only has outgoing edges, while all other vertices have one ingoing edge. This causes differing attachment probabilities on the same tree: Consider for example the tree $T_2$ consisting of two connected vertices and let $\alpha > 0$. In the shape exchangeable case, these two will be indistinguishable while in very simple increasing trees, vertex $1$ will have a larger attachment probability than vertex $2$. Thus the root is harder to distinguish from the other vertices in shape exchangeable trees, making the estimation problem more interesting.
As previously mentioned, each model separates into three subgroups. This is governed by the parameter $\alpha$. $\alpha = 0$ corresponds to uniform attachment, as the vertex (out)degree is not taken into consideration and each vertex is assigned weight $\frac1n$. For $\alpha > 0$, high vertex (out)degrees correspond to a linearly higher attachment weight, giving linear preferential attachment. Finally, for $\alpha < 0$ we see the opposite behavior. Here we must, as is also done in other work~\cite{DHW21}, restrict ourselves to values that will always give a valid attachment probability throughout the run of the process. Then, $\alpha = -\frac1d$ can be seen as each vertex starting out with $d$ free (outgoing) edges, of which the remaining ones are then uniformly sampled for the next attachment. This is equivalent to uniform attachment on a $d$-ary or $d$-regular tree, respectively. While $\alpha \in \{-\frac11, -\frac12\}$ both appear technically possible, we partially exclude them from the allowed parameters in our study: In the very simple increasing tree model, $\alpha = -1$ generates a long path with the root on one end, whereas $\alpha = -\frac12$ in the shape exchangeable model generates a long path with the root in the middle. Now, intuitively we may see that in the very simple increasing tree model, after the first flip has happened, we have the same process with the opposite root color. In the shape exchangeable model, after a flip has happened on both sides of the root, we again have the same process with the opposite root color. Therefore, for any majority estimation to be successful, the amount of vertices before those flips must be larger than constant order. However, this amount is geometrically distributed in both models and so this does not happen with a significant probability. The difference to the viable models with $\alpha < 0$ is that here, the amount of required flips stays constant with growing $n$ while in the other settings it grows exponentially. $\alpha = -1$ is actually not feasible in the shape exchangeable model, as after $T_2$ is generated, all vertices have attachment weight $0$ and no further vertices can attach.

\begin{definition}[Broadcasting Process]\label{def:broadcasting}
    The \emph{broadcasting process} $(\mathcal{T}_n)_{n \in \N}$ with \linebreak $ \mathcal{T}_n= T_n \times \{-1,1\}^n$ is a combination of a growing tree process $T_n$ and a coloring $\{-1,1\}^n$.
    $\mathcal{T}_{n+1}$ is obtained from $\mathcal{T}_n$ as follows:
    At time $n+1$, the new vertex $n+1$ will first choose its
    parent $p_{n+1}$ according to the attachment distribution given by $T_n$. It will then inherit its parent's color $B_{p_{n+1}}$ with probability $1-q$ and flip to the
    other color with probability $q$ independently of all other vertices and their colors. The parameter $q$ is called the \emph{bit-flipping probability} and one realisation of this process is called a \emph{broadcasting tree}.
\end{definition}

Additionally, vertex $1$ has no parent and is therefore assigned a randomly chosen color at time $1$. By symmetry, we may simply call this color ``red'' and the opposite color ``blue''. 
As mentioned above, the problem that we want to study can be viewed as reconstruction of the color at the root vertex~\cite{ADLV22}. At time $N$, we observe $\cT_N$ without any vertex labels or root, but with all vertex colors present. The estimator we consider is the
\emph{majority estimator}:

\begin{definition}
    The \emph{majority estimator} $\bmaj(N,q)$ is defined on a given broadcasting tree of size $N$ with bit-flipping probability $q$ as follows:
\[\bmaj(N,q) \coloneqq
  \begin{cases}
    \sgn\left( \sum_{u \in T_N} B_u\right) & \text{\upshape if }\sum_{u \in T_N} B_u \neq 0, \\
    \text{\upshape Rad}(\tfrac12) & \text{\upshape otherwise,}
  \end{cases} \addtag
\]
with $\text{\upshape Rad}(\tfrac12)$ a Rademacher($\tfrac12)$-distributed random variable.
\end{definition}

This estimator either outputs the color majority in a broadcasting tree of size $N$ or, if there is a tie, it makes a random
guess.
We are interested in analysing the limiting behavior of the error probability in relation to $q$, that is, \[\Rmaj(q) \coloneqq \limsup_{N \to \infty} \Rmaj(N,q) \coloneqq \limsup_{N \to \infty} \Prob*{\bmaj(N,q) \neq B_1}. \addtag\]

\paragraph{Related results}
This particular reconstruction problem has been previously investigated by~\cite{ADLV22} on a subgroup of very simple increasing trees, namely uniform and linear preferential attachment trees, as well as on uniformly grown k-DAGs~\cite{BDL23}. We aim to complete the picture given so far and to provide a more model-agnostic approach to the problem.
The broadcasting process and root color reconstruction have also been investigated on a wide range of random tree~\cite{ACGP21,ADLV22,DHW21,EKPS00,GRP20} and random graph models~\cite{BDL23,MMP20}. For some statistical hardness results for the reconstruction of the root color from the leaf bits, we refer to~\cite{Eft14,HM24,KM22,MSS23}. 

Further, our problem is naturally linked to root-finding algorithms. On uniform and (nonlinear) preferential attachment trees,~\cite{BDL17} showed that there is a vertex set of constant size that contains the root with high probability. This was further extended to uniform attachment on $d$-regular trees in~\cite{KL17}, with a new sharpness result presented for uniform and $d$-regular trees in~\cite{AFKLT24}.
Additionally, a more generally applicable approach to such inference problems has been studied for shape exchangeable trees in~\cite{CX21}. Root reconstruction is also linked to the question of how a given finite seed graph influences the shape and structure of the resulting tree or graph. This problem is studied in~\cite{BEMR17,LP19,RD19} for uniform attachment trees and in~\cite{BMR15,CDKM15} for preferential attachment trees. For general networks this may evolve into studying \emph{hubs} or the position of a central vertex, see \cite{JL17,BB21,BB22}.
Similar problems can be investigated on the stochastic block model~\cite{MNS14,ACGP21} and in models arising from statistical physics, such as the Ising model~\cite{EKPS00,BRZ95}. 
%Here, the root vertex is usually known but one only has partial information on the coloring of the vertices, see for example~\cite{EKPS00,BRZ95}.

Finally, the question of the color majority is closely connected to other stochastic processes that exhibit similar self-interacting behavior. Two processes we will use in this article are Pólya urns~\cite{EP23,Jans04,Mah08} and inhomogeneous random walks~\cite{Hui17,MV08,MPW16}. In a Pólya urn, we may represent the colored vertices as colored balls. The random walk model we consider has time- and space-inhomogeneous increments with vanishing drift. There is a large body of literature on such random walks, for example with non-identically distributed increments~\cite{DSW18} or with drift vanishing at infinity, also known as Lamperti problem~\cite{DKW13,Lam60}. The inhomogeneous random walk model we will use is the (reinforced) Elephant Random Walk (ERW)~\cite{ST04,Lau22,CL23}: Here, the one-dimensional walker remembers a randomly chosen point in the past before each step. With probability $1-q$, it repeats this past step and with probability $q$ it moves in the opposite direction. By representing each colored vertex as the time at which it was added and the color as either ``up'' oder ``down'', the relation to the broadcasting process is quite natural. We will further detail both these representations in~\cref{subsec:polya,subsec:rw}. Both the ERW and the Pólya urn exhibit phase transitions from a (sub-)diffusive to a superdiffusive regime~\cite{Jans04,ST04}. These views of the color majority process then imply such a phase transition depending on the tree parameter $\alpha$ and the bit-flipping probability $q$: If $q$ is too large, the process is \emph{diffusive}, while it is \emph{superdiffusive} for small values of $q$~\cite{Lau22,GLR24}.

\paragraph{Outline of the article}
In~\cref{sec:results} we present our modelling approaches for the majority estimator and our results on its performance in relation to the bit-flipping probability $q$, which we then prove in~\cref{sec:polya,sec:randomwalk}.

% \paragraph{Acknowledgement.} We thank Stephan Wagner for useful discussions on the broadcasting problem and on \cite{DW19}. RS thanks C\'ecile Mailler for useful discussions on Elephant Random Walks. This research is supported by the internal research funding (Stufe I) at Johannes Gutenberg-University, the TOP-ML project and the  Deutsche Forschungsgemeinschaft (DFG, German Research Foundation) through Project-ID 233630050 - TRR 146.

% Local Variables:
% mode: latex
% TeX-engine: luatex
% TeX-master: "../main.tex"
% End:
\section{Results and preliminaries}\label{sec:results}
\subsection{Main results}

\begin{theorem}\label{thm:bmaj-impossibility}
    Let \[f(\alpha) = \begin{cases} \frac{\alpha + 1}{4} & \text{ for very simple inc. trees} \\
    \frac{2 \alpha + 1}{4(\alpha + 1)} & \text{ for shape ex. trees.}
    \end{cases} \addtag\] For shape exchangeable and very simple increasing trees, it holds that for $q \geq f(\alpha)$,
    \[\Rmaj(q) = \frac12. \addtag\]
\end{theorem}

\begin{theorem}\label{thm:bmaj-bound}
For shape exchangeable and very simple increasing trees, it holds that for each
  allowed $\alpha$ there exists $c_\alpha > 0$ such that
  \[\Rmaj(q) \leq c_\alpha \sqrt{q}. \addtag\]
\end{theorem}

\subsection{Discussion of our results}

We note that the lower bound on $q$ in~\cref{thm:bmaj-impossibility} looks quite different for shape exchangeable and very simple increasing trees. In particular, for $\alpha > 3$, $q$ is always smaller than $f(\alpha)$ in the very simple increasing tree model, while that cannot happen in the shape exchangeable model. This highlights that while our two tree models appear very similar, there are some subtle differences. Comparing our impossibility result to~\cite{ADLV22,Lau22_rein}, we find that their results match ours for a subgroup of the considered tree models and it is therefore a natural extension of their results.
In~\cite{ADLV22}, the authors also analyse the majority estimator on very simple increasing trees for $\alpha \geq 0$. Via moment calculations, they achieve a version of~\cref{thm:bmaj-bound} with an error bound on $\Rmaj$ of linear order in $q$, which is sharper than our result. In our approach, we consider the majority estimator as a random walk with memory, also a subject of recent research interest, see e.g.~\cite{Lau22_rein,Ber22,Ber25,DK25}. With this connection, we are able to give a less model-specific analysis of the entire family of very simple increasing and shape exchangeable trees at once. Additionally, we can see in our analysis that the initial phase of the process is essential in determining the long-term behavior, even though our restriction to the setting where the process first crosses a very high boundary is at fault for our weaker error bound. Since the random walk model exhibits a vanishing drift term, expecting this behavior still seems reasonable and supports the intuitive understanding of the process. Finally, expanding the study to the $\alpha < 0$ range lets us consider trees with a given maximal (out)degree and it is interesting to see that root color estimation is still viable even in restricted-degree models. Further expansions to other random graph models also seem possible as long as they fit into the random walk with memory viewpoint.

\subsection{Color majority as an inhomogeneous random walk}\label{subsec:rw}
Calculating the color majority of a broadcasting tree does not require any information about the tree structure, only the vertex colors. Therefore, we may consider the process describing the evolution of the \emph{color difference}~\cite{ADLV22}.

\begin{definition}
    Call the color of the root vertex ``red''. Let $\#_{\text{\upshape red}}(n)$ and $\#_{\text{\upshape blue}}(n)$ denote the number of red, respectively blue, vertices at time $n$. Then set
    \begin{equation}
    \Delta_1(n) \coloneqq \#_{\text{\upshape red}}(n) - \#_{\text{\upshape blue}}(n) \quad \text{ for all } n \leq N,
    \end{equation}
    with $\Delta_1(1) = 1$.
\end{definition}

It is clear that in each timestep, the color difference may only increase or decrease by exactly one. In the $\alpha = 0$  case, the current color difference is sufficient to describe the distribution of these increments. For $\alpha \neq 0$, the weights of the vertices must also be taken into consideration.

\begin{definition}
    Let
    \begin{align*}
        \#_{\text{\upshape red weight}}(n)={}& \begin{cases}
            \sum_{v\text{ \upshape red}} \degpl{n}{v} & \text{\upshape for very simple inc. trees} \\
            \sum_{v\text{ \upshape red}} \deg{n}{v} & \text{\upshape for shape ex. trees} \\
        \end{cases} \\
    \intertext{as well as}
        \#_{\text{\upshape blue weight}}(n)={}& \begin{cases}
            \sum_{v\text{ \upshape blue}} \degpl{n}{v} & \text{\upshape for very simple inc. trees} \\
            \sum_{v\text{ \upshape blue}} \deg{n}{v} & \text{\upshape for shape ex. trees.} \\
        \end{cases}
    \end{align*}
\end{definition}

In very simple increasing trees, e.g. $\#_{\text{\upshape red weight}}(n)$ is given by the number of outgoing edges that the red vertices have, while in shape exchangeable trees it is given by the total number of edges that the red vertices have.

\begin{definition}
    Let
    \[\Delta_2(n) \coloneqq \#_{\text{\upshape red weight}}(n) - \#_{\text{\upshape blue weight}}(n) \quad \text{ for all } n \leq N, \addtag\]
    with $\Delta_2(1) = 0$.
\end{definition}

 With these two processes, we can now completely describe the evolution of the color difference in the broadcasting process.

\begin{definition}\label{def:delta-n}
    Let
    \[\Delta(n) \coloneqq \colvec{\Delta_1(n)}{\Delta_2(n)}, \addtag\]
    with
    \[\Delta_1(1) = 1 \qquad \Delta_2(1) = 0 \qquad \text{ and } \qquad \Delta(n+1) = \Delta(n) + D(n). \addtag\]
    In very simple increasing trees, 
    \[D(n) \in \Set*{\colvec{1}{1}, \colvec{1}{-1}, \colvec{-1}{1}, \colvec{-1}{-1}}\]
    and in shape exchangeable trees,
    \[D(n) \in \Set*{\colvec{1}{2}, \colvec{1}{0}, \colvec{-1}{0}, \colvec{-1}{-2}},\]
    each tuple corresponding to the attachment of a red, respectively blue, vertex to an existing red, respectively blue, vertex.
\end{definition}

Specifically, $\Delta_1(n)$ and $\Delta_2(n)$ together form a
two-dimensional time-inhomogeneous Markov random walk.

\begin{definition}
    We define the normalisation $\Zan$:
    \begin{equation}
    \Zan \coloneqq
    \begin{cases}
      \alpha(1-\frac1n)+1 & \text{ for very simple inc. trees} \\
      2\alpha(1-\frac1n)+1 & \text{ for shape ex. trees.}
    \end{cases}
\end{equation}
\end{definition}

\begin{lemma}
In both tree models, the probability that the new vertex $n+1$ attaches to an existing red vertex is
\begin{equation}
\Prob*{n+1 \sim \text{\upshape red vertex} \given \cT_n}= \frac12\left(1 + \frac{\Delta_1(n) + \alpha \Delta_2(n)}{\Zan n} \right).
\end{equation}
\end{lemma}

\begin{proof}
For very simple increasing trees the attachment distribution is given by:
\begin{align*}
  \Prob*{n+1 \sim \text{red vertex} \given \cT_n} ={}& \frac{\sum_{v\text{ red}}(\alpha \degpl{n}{v} + 1)}{\sum_{u \in \cT_n}(\alpha \degpl{n}{u} +1)} \\
  ={}& \frac{\sum_{v\text{ red}}(\alpha \degpl{n}{v} + 1)}{\alpha(n-1) + n} \\
  ={}& \frac12\left(1 + \frac{\Delta_1(n) + \alpha \Delta_2(n)}{(\alpha(1-\frac1n)+1)n} \right) \eqqcolon p_{\text{vs}}(\alpha,n). \addtag
\end{align*}
Similarly, for shape exchangeable trees, 
\begin{align*}
  \Prob*{n+1 \sim \text{red vertex} \given \cT_n} ={}& \frac12\left(1 + \frac{\Delta_1(n) + \alpha \Delta_2(n)}{(2\alpha(1-\frac1n)+1)n} \right) \eqqcolon p_{\text{se}}(\alpha,n). \addtag
\end{align*}
\end{proof}

All in all, let $\cF(n)$ be the natural filtration of $\Delta(n)$. Then, in very
simple increasing trees $D(n)$ is distributed as
\begin{align*}
  \Prob*{D(n) = \colvec{1}{1} \given \cF(n)} ={}& p_{\text{vs}}(\alpha,n) \cdot (1 - q) \\
  \Prob*{D(n) = \colvec{-1}{1} \given \cF(n)} ={}& p_{\text{vs}}(\alpha,n) q \\
  \Prob*{D(n) = \colvec{1}{-1} \given \cF(n)} ={}& \left(1 - p_{\text{vs}}(\alpha,n)\right)\cdot q\\
  \Prob*{D(n) = \colvec{-1}{-1} \given \cF(n)} ={}& \left(1 - p_{\text{vs}}(\alpha,n)\right)\cdot (1-q)  \addtag
\end{align*}
and in shape exchangeable trees as
\begin{align*}
  \Prob*{D(n) = \colvec{1}{2} \given \cF(n)} ={}& p_{\text{se}}(\alpha,n)\cdot(1-q) \\
  \Prob*{D(n) = \colvec{-1}{0} \given \cF(n)} ={}& p_{\text{se}}(\alpha,n)\cdot q \\
  \Prob*{D(n) = \colvec{1}{0} \given \cF(n)} ={}& (1 - p_{\text{se}}(\alpha,n))\cdot q \\
  \Prob*{D(n) = \colvec{-1}{-2} \given \cF(n)} ={}&  (1 - p_{\text{se}}(\alpha,n))\cdot (1-q). \addtag
\end{align*}

Finally, we relate the behavior of this inhomogeneous random walk to our estimation problem.

\begin{lemma}\label{thm:Rmaj-Delta}
\begin{equation}
\Rmaj(q) \leq \limsup_{N \to \infty} \Prob*{\Delta_1(N) \leq 0}.
\end{equation}
\end{lemma}

\begin{proof}
If $\Delta_1(N)$ is negative, the majority estimator on $N$ vertices is wrong. If $\Delta_1(N)$ is zero, it takes a random guess. This holds for any $N \in \N_1$ and the claim follows via
% We observe that 
% \begin{align}
% \Delta_1(N) < 0 \implies{}& \bmaj(N,q) \neq B_1\\
% \intertext{and}
% \Delta_1(N) = 0 \implies{}& \bmaj(N,q) = \text{\upshape Rad}(\tfrac12),
% \end{align}
% which implies
\[\Rmaj(N,q) =  \Prob*{\Delta_1(N) < 0} + \frac{1}{2}\Prob*{\Delta_1(N) = 0} \leq \Prob*{\Delta_1(N) \leq 0}. \addtag\]
\end{proof}

The random walk $(\Delta_1(n))_{n \in \N_1}$ is not only inhomogeneous in both time and space, it also exhibits vanishing drift.
We have
\[\Exp*{D_1(n)} = (1-2q) \frac{\Delta_1(n) + \alpha \Delta_2(n)}{\Zan n}.\]
First note that if $q=\frac12$, $\Exp*{D_1(n)} = 0$ for all $n$ and $\Delta_1(n)$ performs a simple random walk. Then, \[\limsup_{N \to \infty} \Prob*{\Delta_1(N) \leq 0} = \frac12.\] If $q \neq \frac12$, it determines the direction of the drift in combination with the current sign of $\Delta_1(n)+\alpha\Delta_2(n)$. A small $q$ induces a self-reinforcing drift in the same direction as $\sgn(\Delta_1(n)+\alpha\Delta_2(n))$, whereas a large $q$ will cause a self-weakening effect. The actual position of $\Delta_1(n)+\alpha \Delta_2(n)$ controls the strength of the drift. However, note that $\abs*{\Delta_1(n)+\alpha\Delta_2(n)} \leq \Zan n$ - and in fact, $\abs*{\Delta_1(n)+\alpha\Delta_2(n)} \in o(\Zan n)$ with high probability. In this situation, $\Exp*{D_1(n)}$ will converge to zero as $n$ grows towards infinity and we must control the speed at which this drift disappears. We achieve this via a supermartingale argument first presented by Menshikov and Volkov~\cite{MV08}, where we show that with high probability, we enter an escape regime in which the drift vanishes slowly enough for superdiffusive behavior to occur.

\iffalse
\begin{align*}
  \Prob*{D(n) = \colvec{1}{1} \given \cF_n} ={}& \frac12\left(1 + \frac{\Delta_1(n) + \alpha \Delta_2(n)}{\Zan n} \right)\cdot(1-q) \\
  \Prob*{D(n) = \colvec{-1}{1} \given \cF_n} ={}& \frac12\left(1 + \frac{\Delta_1(n) + \alpha \Delta_2(n)}{\Zan n} \right)\cdot q \\
  \Prob*{D(n) = \colvec{1}{-1} \given \cF_n} ={}& \left(1-\frac12\left(1 + \frac{\Delta_1(n) + \alpha \Delta_2(n)}{\Zan n} \right)\right)\cdot q \\
  \Prob*{D(n) = \colvec{-1}{-1} \given \cF_n} ={}& \left(1-\frac12\left(1 + \frac{\Delta_1(n) + \alpha \Delta_2(n)}{\Zan n} \right)\right)\cdot (1-q), \addtag
\end{align*}
%
and in shape exchangeable trees, it is distributed as
\begin{align*}
  \Prob*{D(n) = \colvec{1}{2} \given \cF_n} ={}& \frac12\left(1 + \frac{\Delta_1(n) + \alpha \Delta_2(n)}{\Zan n} \right)\cdot(1-q) \\
  \Prob*{D(n) = \colvec{-1}{0} \given \cF_n} ={}& \frac12\left(1 + \frac{\Delta_1(n) + \alpha \Delta_2(n)}{\Zan n} \right)\cdot q \\
  \Prob*{D(n) = \colvec{1}{0} \given \cF_n} ={}& \left(1-\frac12\left(1 + \frac{\Delta_1(n) + \alpha \Delta_2(n)}{\Zan n} \right)\right)\cdot q \\
  \Prob*{D(n) = \colvec{-1}{-2} \given \cF_n} ={}& \left(1-\frac12\left(1 + \frac{\Delta_1(n) + \alpha \Delta_2(n)}{\Zan n} \right)\right)\cdot (1-q). \addtag
\end{align*}
\fi

\subsection{Color majority as a Pólya urn with randomized replacement}\label{subsec:polya}

One can also imagine the above process as a \emph{Pólya urn} with randomized replacement: At each timestep, we draw a ball of one color from the urn and add a new ball of the same color with probability $1-q$ or one of the opposite color with probability $q$.
To formalize this, we use the notation from~\cite{Jans04} which is also used by e.g.~\cite{ADLV22,DHW21}. 
Generally, an $m$-type Pólya urn process is given by
\[X(n) = \big( (X_{i,n})_{i=1}^m\big)_ {n \in \N},\] where $X_{i,n}$ is the random
variable describing the amount of balls of type $i$ in the urn at time $n$.
The evolution of the Pólya process is given by the \emph{replacement vectors} $\xi_j$ - if a ball of type $j$ is drawn at time $n$, then
\[X_{i,n+1} = X_{i,n} + \xi_{j,i}.\]
Again, we need to represent both the amount of vertices of each color and the respective attachment weights to have the full picture. In the Pólya urn model we achieve this by associating two types to each color, a \emph{weight type} and a \emph{count type}. The weight types $\{r_w, b_w\}$ should fulfill that the total amount of balls of one type is proportional to the entire attachment weight of the represented color, while the count types $\{r_c,b_c\}$ count the red, respectively blue, vertices. As~\cite{DHW21}, we set the
activities of these types to $a_{r_w} = a_{b_w} = 1$, $a_{r_c} = a_{b_c} =
0$ and number them $r_w = 1$, $b_w = 2$, $r_c = 3$, $b_c = 4$. As is well known, the \emph{(expected) replacement matrix} $A$ given by
\[A \coloneqq \left(a_j\Exp*{\xi_{j,i}}\right)_{i,j}\] is
quite important for the analysis of the Pólya urn process. Note that we put the expected replacement vectors in the columns of the matrix as in~\cite{Jans04}.
For very simple increasing trees, this associated Pólya urn has the following
expected replacement matrix~\cite{DHW21,ADLV22}:
\[\addtag \label{eq:outdeg-replmatrix}
  A_{\text{vs}} = \begin{pmatrix*}[c]
    \alpha +1 -q & q & 0 & 0 \\
    q & \alpha +1 - q & 0 & 0 \\
    1 -q & q & 0 & 0 \\
    q & 1-q & 0 & 0
  \end{pmatrix*}
\]
and initial vector $X(0) = (1, 0, 1, 0)$.
For shape exchangeable trees, we follow the same modelling idea, but with a
slightly different replacement rule: After time $1$, each new vertex starts with attachment weight $\alpha +1$ instead of $1$, which changes the expected replacement matrix to
\[\addtag \label{eq:deg-replmatrix}
  A_{\text{se}} = \begin{pmatrix*}[c]
    \alpha +(1-q)(\alpha+1) & q(\alpha+1) & 0 & 0 \\
    q(\alpha+1) & \alpha +(1-q)(\alpha+1) & 0 & 0 \\
    1 -q & q & 0 & 0 \\
    q & 1-q & 0 & 0.
  \end{pmatrix*}
\]
The initial vector $X(0) = (1, 0 , 1, 0)$ is the same since the root has degree $0$ at time $1$. Then, for both tree models, we can write $\Delta(n)$ from~\cref{def:delta-n} as
\[\Delta(n) = \bigcolvec{X_{3,n} - X_{4,n}}{(X_{1,n} - X_{3,n}) - (X_{2,n} - X_{4,n})}. \addtag\]

\cref{thm:bmaj-impossibility} is proven in~\cref{sec:polya}. The proof of~\cref{thm:bmaj-bound} is given in~\cref{sec:randomwalk}.

% Local Variables:
% mode: latex
% TeX-engine: luatex
% TeX-master: "../main.tex"
% End:
\section{Proving~\texorpdfstring{\cref{thm:bmaj-impossibility}}{Theorem 5}}\label{sec:polya}
To apply the convergence results from~\cite[Thms.
  3.22-3.24]{Jans04} (see also~\cite[Thm. 3.1]{DHW21}), we need to check that the expected replacement matrices fulfill the necessary assumptions~\cite[(A1)-(A8)]{DHW21}:

\begin{lemma}\label{thm:polya-conv}
The Pólya urns described in~\cref{subsec:polya} with expected replacement matrices given in~\cref{eq:outdeg-replmatrix,eq:deg-replmatrix} exhibit the following convergence behavior: Let $\lambda_1 > \lambda_2$ be the first two eigenvalues. Then,
  \begin{enumerate}
  \item \label{eq:polyaconv-1} if $\lambda_1 = 2\lambda_2$:
    \[\frac{X(N) - N\lambda_1v_1}{\sqrt{N\ln(N)}}
      \conv{d} \mathcal{N}(0, \Sigma_I) \addtag \]
  \item \label{eq:polyaconv-2} and if $\lambda_1 > 2\lambda_2$:
    \[\frac{X(N) - N\lambda_1v_1}{\sqrt{N}}
      \conv{d} \mathcal{N}(0,\Sigma_{II}) \addtag \]
  \end{enumerate}
  with $\Sigma_{I,II}$ as defined in~\cite{Jans04} or~\cite[Section 3]{DHW21}.
\end{lemma}

\begin{proof}
If $q \neq \frac{\alpha + 1}{2}$, $A_{\text{vs}}$ is diagonalizable with eigenvalues $\Lambda= \{\alpha + 1, {\alpha + 1 - 2q}, 0, 0\}$. If \mbox{$q = \frac{\alpha+1}{2}$,} $A_\text{vs}$ has eigenvalues $\{\alpha+1, 0, 0\}$. Similarly, $A_\text{se}$ is diagonalizable for $q \neq \frac{2\alpha + 1}{2(\alpha +1)}$ with eigenvalues \mbox{$\Lambda = \{2\alpha+1, 2\alpha + 1-2q-2\alpha q, 0, 0\}$.} If $q = \frac{2\alpha +1}{2(\alpha+1)}$, then $\Lambda = \{\frac{2\alpha^2 + 3\alpha + 1}{\alpha +1}, 0, 0\}$. One easily checks that the remaining conditions hold for both matrices (see also~\cite{DHW21}). 
\end{proof}

As we can see in~\cref{thm:polya-conv}, the ratio between $\lambda_1$ and $\lambda_2$ is essential in determining the convergence
behavior. Note that  
\begin{equation}
  \label{eq:lambda-q-relation}
  \lambda_1 \geq 2\lambda_2 \iff q \geq f(\alpha) = \begin{cases} \frac{\alpha + 1}{4} & \text{ for very simple inc. trees} \\
    \frac{2 \alpha + 1}{4(\alpha + 1)} & \text{ for shape ex. trees.}
    \end{cases}
\end{equation}
Additionally, we remark that whether the matrices are diagonalizable or not, the first right eigenvector, $v_1$, always fulfills $v_{1,3} = v_{1,4} = 1$.
With this,~\cref{thm:bmaj-impossibility} follows directly:

\begin{proof}[Proof of~\cref{thm:bmaj-impossibility}]
  For $q \geq f(\alpha)$ let
   \[g(N) = \begin{cases}
      \sqrt{N} & \text{ if $\lambda_1 = 2\lambda_2$} \\
      \sqrt{N \ln(N)} & \text{ if $\lambda_1 > 2\lambda_2$}
  \end{cases} \addtag \]
  and define
  \begin{equation}
    \widetilde{X}_{3,N} \coloneqq{} \frac{X_{3,N} -N\lambda_1}{g(N)},\quad \widetilde{X}_{4,N} \coloneqq{} \frac{X_{4,N} - N\lambda_1}{g(N)},
  \end{equation}
  Then $(X(N) - N\lambda_1 v_1)/g(N)$ converges jointly to a normal distribution, implying (since $v_{1,3} = 1 = v_{1,4}$ as mentioned above)
  % and by the projection property of normal distributions we get from~\cref{thm:polya-conv}
\[\colvec{\widetilde{X}_{3,N}}{\widetilde{X}_{4,N}} \conv{d} \colvec{\widetilde X_3}{\widetilde X_4} \sim \mathcal{N}(0, \Sigma') \addtag\]
where calculating the covariance matrices $\Sigma_{I,II}$ gives
  \[\Sigma' = \sigma(\alpha,q)
    \begin{pmatrix*}[c]
      1 & -1 \\
      -1 & 1
    \end{pmatrix*} \addtag
  \]
  %and $\sigma(\alpha, q) > 0$ - the structure of $\Sigma'$ for shape-exchangeable trees follows from a similar calculation as given for very simple increasing trees in~\cite{DHW21}. 
  for both tree models. With this covariance structure, $\widetilde{X}_3 - \widetilde{X}_4$ is also normal-distributed with mean $0$. Analogous to~\cref{thm:Rmaj-Delta}, it holds that \[\liminf_{N \to \infty} \Rmaj(N, q) \geq \liminf_{N \to \infty} \Prob*{\Delta(N) < 0},\] giving~\cite{ADLV22}
  \begin{align*}
    \liminf_{N \to \infty} \Rmaj(N,q) \geq{}& \liminf_{N \to \infty} \Prob*{\Delta(N) < 0} \\
    ={}& \liminf_{N \to \infty} \Prob*{\widetilde X_{3,N} - \widetilde X_{4,N} < 0}\\
  %  ={}& \Prob*{\widetilde{X}_3 - \widetilde{X}_3 < 0} \\
    ={}& \Prob*{\widetilde{X}_3 - \widetilde{X}_4 \leq 0} = \frac12 \addtag \\
  \intertext{and similarly}
    \limsup_{N \to \infty} \Rmaj(N,q) = \Rmaj(q) \leq{}& \limsup_{n \to \infty} \Prob*{\Delta(N) \leq 0} = \frac12. \addtag
  \end{align*}
  Therefore, $\Rmaj(q) = \frac12$ for $q \geq f(\alpha)$ follows.
  %for both very simple increasing and shape exchangeable trees.
\end{proof}

\section{Proving~\texorpdfstring{\cref{thm:bmaj-bound}}{Theorem 6}}\label{sec:randomwalk}
The proof of ~\cref{thm:bmaj-bound} uses the random walk model presented in~\cref{subsec:rw}. We recall~\cref{thm:Rmaj-Delta} and investigate the event $\{\Delta_1(N) > 0\}$, which implies correctness of the majority estimator on $N$ vertices.

\begin{definition}\label{def:parameters2}
Let 
\begin{align*}
    M_2 \coloneqq{}& \max_{\omega}{} (D_1(n) + \alpha D_2(n))^2(\omega). \addtag
\end{align*}
\end{definition}

\begin{definition}\label{def:parameters1}
Set
\[\beta = \begin{cases}
    \alpha + \frac23 & \text{ for very simple inc. trees} \\
    \frac32 \alpha + \frac34 & \text{ for shape ex. trees}
\end{cases}\]
and define, for $\tilde c_\alpha > 0$, 
\[B \coloneqq \sqrt{\frac{3M_2 + \tilde c_\alpha}{\Zan(2\beta - \Zan)}}. \addtag \]
For $\gamma \in (0, \frac12)$, $q \in (0,1]$ set \[A \coloneqq q^{\gamma-1/2} > 1.\]
\end{definition}

\begin{remark}\label{rem:A>B>1}
    Note that for small enough values of $q$, $A > B$ holds. Further, $B$ is well-defined for $n > 1$ since $2\beta - \Zan > 0$ for all $\alpha$ in the allowed ranges of the respective model. Finally, $B > 1$ by choosing $\tilde c_\alpha$ large enough.
\end{remark}

\begin{remark}\label{rem:Bconv}
With
\begin{align*}
& \lim_{N \to \infty} \ZaN \eqqcolon Z_\alpha = \begin{cases}
      \alpha+1 & \text{ for very simple inc. trees} \\
      2\alpha+1 & \text{ for shape ex. trees.}
    \end{cases}, \addtag \\
\intertext{it holds that}
& \lim_{N \to \infty} B =  \sqrt{\frac{3M_2 + \tilde c_\alpha}{Z_\alpha(2\beta - Z_\alpha)}}. \addtag
% & \lim_{N \to \infty} B = \frac{1}{Z_\alpha}\sqrt{(3M_2 + \tilde c_\alpha) \Big/
%       \frac{\alpha+\frac23}{Z_\alpha}} \lor 1. \addtag
\end{align*}
\end{remark}

\begin{definition}
 We define the following stopping times for any $A > B > 1$:
  \begin{align*}
    \taugood(A) \coloneqq{}& \inf \Set*{n > 0 \given \Delta_1(n) + \alpha\Delta_2(n) > A\Zan \sqrt{n}} \\
    \taubad(B) \coloneqq{}& \inf \Set*{n > \taugood \given \Delta_1(n) + \alpha\Delta_2(n) \leq B\Zan \sqrt{n}}. \addtag
  \end{align*}
\end{definition}

With this notation, there exists $N_0>0$ such that for all $N>N_0$ 
\begin{align*}\label{eq:25}
  & \Prob*{\Delta_1(N) > 0} \\ \geq{}& \Prob{\Delta_1(N) > 0 \nonscript\:\vert\nonscript\: \taugood(A) \leq N, \taubad(B) > N}\Prob{\taugood(A) \leq N, \taubad(B) > N}. \addtag
\end{align*}

\begin{figure}
    \begin{subfigure}[c]{0.45\textwidth}
    \includegraphics[width=0.95\textwidth]{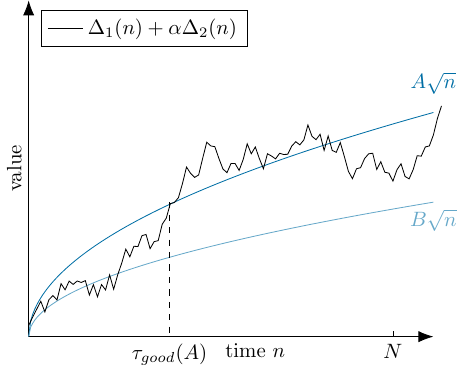}
%        \begin{tikzpicture}
%            \pic{taugood};
%        \end{tikzpicture}
        \caption{}%A path of $\Delta_1(n) + \alpha \Delta_2(n)$ fulfilling $\taugood(A) \leq N, \taubad(B) > N$.}
        \label{subfig:evt2}
    \end{subfigure}
    \hspace*{3em}
    \captionsetup[subfigure]{oneside,margin={1em,0cm}}
    \begin{subfigure}[c]{0.45\textwidth}
    %\raisebox{1.34em}{
    \includegraphics[raise=0.4em,width=0.95\textwidth]{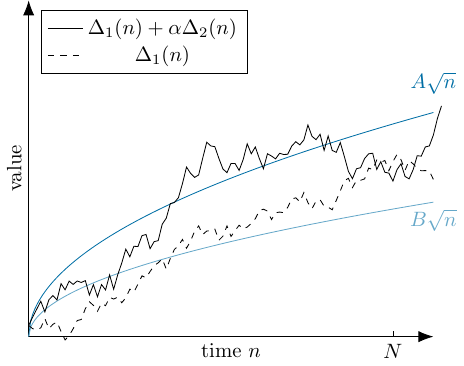}
%        \begin{tikzpicture}
%            \pic{delta1};
%        \end{tikzpicture}%}
        \caption{}%A path of $\Delta(n)$ where, given (a), $\Delta_1(N) > 0$.}
        \label{subfig:evt1}
    \end{subfigure}
    \caption{Illustrating the two events in~\cref{eq:25}. In (a), the combined process ${\Delta_1(n) + \alpha \Delta_2(n)}$ passes the $\taugood(A)$-boundary before the time horizon $N$ and does not drop below the $\taubad(B)$-boundary again. In (b), we have the same behavior as in (a) and additionally see that the isolated process $\Delta_1(n)$ is above zero at time $N$.}
    \label{fig:taugood-taubad}
\end{figure}

\begin{lemma}\label{lem:1-bound}
  For all $q \in (0,1]$ and each allowed value of $\alpha$, there exists $c_{\alpha,1} > 0$ such that \[\liminf_{N \to \infty}{}\: \Prob{\Delta_1(N) > 0 \given \taugood(A) \leq N, \taubad(B) > N} \geq 1 - c_{\alpha,1}\sqrt{q}. \addtag \]
\end{lemma}

\begin{lemma}\label{lem:2-bound}
  For  all $q \in (0,1]$ and each allowed value of $\alpha$, there exists $c_{\alpha,2} > 0$ such that \[\liminf_{N \to \infty}{}\: \Prob{\taugood(A) \leq N, \taubad(B) > N} \geq 1 -
    c_{\alpha,2}\sqrt{q}. \addtag \]
\end{lemma}

Together,~\cref{lem:1-bound,lem:2-bound} imply our theorem.

\begin{proof}[Proof of~\cref{thm:bmaj-bound}]
Firstly, note that for $q = 0$, it holds that $\Delta_1(n) = n$ for all $n \in \N$ and the theorem follows immediately. For $q > 0,$
  \begin{align*}
    &\liminf_{N \to \infty} \Prob*{\Delta_1(N) > 0} \\ \geq{}& \liminf_{N \to \infty} \left(\Prob{\Delta_1(N) > 0 \given \taugood(A) \leq N, \taubad(B) > N}\cdot \Prob{\taugood(A) \leq N, \taubad(B) > N} \right) \\
    \geq{}& (1-c_{\alpha,1}\sqrt{q})(1-c_{\alpha,2}\sqrt{q}) \\
    \geq{}& 1 - c_\alpha\sqrt{q}. \addtag
  \end{align*}
  With $1 - \Rmaj(q) \geq \liminf_{N \to \infty} \Prob*{\Delta_1(N) > 0}$, the claim follows.
\end{proof}

\begin{remark}
    By keeping track of the constants needed in the proofs, we find that $c_\alpha$ is increasing in $\alpha$ for $\alpha > 0$. Intuitively, the difference between $\Delta_1$ and $\alpha \Delta_2$ grows with $\alpha$, making it harder to control $\Delta_1$ based on the sum of the two processes.
\end{remark}

In the following subsections, we first prove~\cref{lem:2-bound} and then~\cref{lem:1-bound}.
\subsection{Proving~\texorpdfstring{\cref{lem:2-bound}}{Lemma 13}}\label{sec:lem-supmart}
To prove~\cref{lem:2-bound}, we adopt a line of argumentation presented by Menshikov
and Volkov~\cite{MV08} and consider the auxilliary process
\[Y(n) \coloneqq \frac{n}{(\Delta_1(n) + \alpha\Delta_2(n))^2}, \addtag \]
which is adapted to $\cF(n)$, the filtration generated by $\Delta(n)$.

\begin{lemma}\label{lem:Y-supmart}
  For any $\gamma \in (0, \frac12)$ there
  exists a threshold $q_0$ such that for all ${0 < q < q_0}$:
  $A > B >
  1$ and the stopped process $Y(\taugood(A)
  \lor n \land \taubad(B))$ is a nonnegative supermartingale on $\cF(n)$.
\end{lemma}

\begin{proof}
By~\cref{rem:A>B>1}, we have a $\hat q_0$ such that for all $0 < q < \hat q_0$, $A > B > 1$. For such $q$, the event that $\taugood(A) \leq N$ and $\taubad(B) > N$ is as indicated by~\cref{subfig:evt2}. As we only consider ${\taubad(B) \geq n\geq \taugood(A)}$, the process ${\Delta_1(n) + \alpha\Delta_2(n)}$ is nonzero and $Y(n)$ is well-defined on these $n$. We prove \[\Exp*{Y((\taugood(A) \lor n+1 \land \taubad(B))-Y((\taugood(A)
  \lor n \land \taubad(B)) \given \cF(n)} \leq 0\]
for sufficiently small values of $q$. 

If $n \land \taubad(B) = \taubad(B)$,
\begin{align*}
  Y( (n+1) \land \taubad(B)) - Y(n \land \taubad(B)) ={}& Y(\taubad(B)) - Y(\taubad(B)) = 0. \addtag 
\end{align*}
For $n \land \taubad(B) = n$,
\begin{align*}
  &Y((n+1) \land \taubad(B)) - Y(n \land \taubad(B))\\ ={}& Y(n+1) - Y(n) \\
  ={}&\frac{n+1}{(\Delta_1(n+1) + \alpha\Delta_2(n+1))^2} - \frac{n}{(\Delta_1(n) + \alpha\Delta_2(n))^2} \\
  ={}& \frac{n+1}{(\Delta_1(n) + D_1(n) + \alpha\Delta_2(n) + \alpha D_2(n))^2} - \frac{n}{(\Delta_1(n) + \alpha\Delta_2(n))^2} \\
  ={}& \frac{1}{(\Delta_1(n)+\alpha\Delta_2(n))^2} \left( \frac{n+1}{(1 + \frac{D_1(n) + \alpha D_2(n)}{\Delta_1(n)+\alpha\Delta_2(n)})^2} - n\right) \\
  ={}& \frac{n+1}{(\Delta_1(n)+\alpha\Delta_2(n))^2} \left( \frac{1}{(1 + \frac{D_1(n) + \alpha D_2(n)}{\Delta_1(n)+\alpha\Delta_2(n)})^2} - \frac{n}{n+1}\right). \addtag
\end{align*}
Taking the conditional expectation with respect to $\mathcal{F}(n)$, we have
\begin{align*}\label{Eq.lisa.2}
  & \Exp*{Y(n+1) - Y(n) \given\mathcal{F}(n)}\\ 
  ={}& \frac{n+1}{(\Delta_1(n)+\alpha\Delta_2(n))^2} \Exp*{\frac{1}{(1 + \frac{D_1(n) + \alpha D_2(n)}{\Delta_1(n)+\alpha\Delta_2(n)})^2} - \frac{n}{n+1} \given  \mathcal{F}(n)}. \addtag
\end{align*}
Note that the first factor is always positive. 
Set 
\begin{equation}
 f(x) \coloneqq{} \frac{1}{(1 + x)^2} - \frac{n}{n+1} \text{ for $x > -1$.}
\end{equation}
By a second order Taylor expansion around zero,
%Taking a closer look (by
 % considering it as a function of
%  $\frac{D_1(n)+\alpha D_2(n)}{\Delta_1(n)+\alpha\Delta_2(n)}$ and doing a Taylor expansion
  %up to the second order around $0$), reveals:
% 
\begin{align*}
    & f(x) = 1 - \frac{n}{n+1} - 2x + 3x^2 + R_2f(x; 0), \addtag
  \end{align*}
where we bound $R_2f(x;0)$ uniformly for $x > -1$: If $x \geq 0$, there exists $\xi \in [0, x]$ such that
  \begin{align*}
    & R_2f(x;0) = - 24 \left( 1 + \xi \right)^{-5}x^3 \leq 0. \addtag
  \end{align*}
If $-1 <{} x < 0$, there exists $\xi \in [x, 0] \subset (-1,0]$ such that
 \begin{equation}\label{Eq.lisa.1}
     R_2f(x;0) 
    ={} - 24 \left( 1 + \xi \right)^{-5}x^3.
    \end{equation}
With $x = \frac{D_1(n) + \alpha D_2(n)}{\Delta_1(n) + \alpha \Delta_2(n)}$, for $\taugood(A) \leq n \leq \taubad(B)$ it holds that $x$ is bounded away from $-1$ and there exists $\tilde c_\alpha
  > 0$ such that \eqref{Eq.lisa.1} is bounded from above by
  %
  %  Let $\xi \in [\frac{D_1(n) + \alpha
  %  D_2(n)}{\Delta_1(n) + \alpha \Delta_2(n)}, 0] \subset (-1+c,0) \cup \{0\}$
 % (Note that $n > 1$ by definition of $\taugood$).
 % We now bound the remainder term uniformly for all considered values of $n$ and
  %all realisations of $\Delta$. 
%
%Recall that $\taugood(A) \leq n \leq \taubad(B)$.
%
%  
 %
  \begin{align*}
     \tilde c_\alpha \left(\frac{1}{\Delta_1(n) + \alpha \Delta_2(n)}\right)^3 
    \leq{}& \tilde c_\alpha \left(\frac{1}{\Delta_1(n) + \alpha \Delta_2(n)}\right)^2, \addtag
  \end{align*}
  where the upper bound follows because ${\Delta_1(n) + \alpha \Delta_2(n) > 0}$ almost surely for $n$ between $\taugood(A)$ and $\taubad(B)$. Thus, for all such $n$ and all
  realisations of $\Delta(n)$ (a.s.),
  \begin{align*}
&f\left(\frac{D_1(n)+\alpha D_2(n)}{\Delta_1(n)+\alpha\Delta_2(n)}\right) \\\leq{}&   1 - \frac{n}{n+1} - 2\frac{D_1(n)+\alpha D_2(n)}{\Delta_1(n)+\alpha\Delta_2(n)} + 3\left( \frac{D_1(n)+\alpha D_2(n)}{\Delta_1(n)+\alpha\Delta_2(n)} \right)^2 +  \tilde c_\alpha \left(\frac{1}{\Delta_1(n) + \alpha \Delta_2(n)}\right)^2. \addtag
  \end{align*}
Hence, the conditional expectation on the righthandside of \eqref{Eq.lisa.2} is bounded from above by
\begin{align*}\label{eq:condexp}
 &   \Exp*{ \frac{1}{n+1} - 2\frac{D_1(n)+\alpha D_2(n)}{\Delta_1(n)+\alpha\Delta_2(n)} + 3\left( \frac{D_1(n)+\alpha D_2(n)}{\Delta_1(n)+\alpha\Delta_2(n)} \right)^2 + \tilde  c_\alpha \left(\frac{1}{\Delta_1(n) + \alpha \Delta_2(n)}\right)^2 \given \mathcal{F}(n)} \nonumber \\
    & 
=    \frac{1}{n+1} - \frac{2 \Exp*{D_1(n) + \alpha D_2(n) \given \cF(n)}}{\Delta_1(n) + \alpha \Delta_2(n)}
    + \frac{3 \Exp*{(D_1(n)+\alpha D_2(n))^2 \given \cF(n)}}{(\Delta_1(n)+\alpha\Delta_2(n))^2}\\
    & \quad + \tilde c_\alpha \left(\frac{1}{\Delta_1(n) + \alpha \Delta_2(n)}\right)^2. \addtag
\end{align*}
The first moment of $D_1(n)+\alpha D_2(n)$ given $\cF(n)$ is 
  \begin{equation}
      \frac{\rho}{\Zan}\frac1n (\Delta_1(n) + \alpha \Delta_2(n)),
  \end{equation}
  with $\rho$ defined as
  \[\rho \coloneqq
    \begin{cases}
      \alpha + 1 - 2q & \text{ for very simple increasing trees} \\
      2\alpha + 1 - 2q(\alpha +1) & \text{ for shape exchangeable trees}
    \end{cases} \addtag
  \]
  and the second moment is bounded from above by $M_2$.
  Therefore,~\cref{eq:condexp} is bounded from above by
\begin{equation}
    \frac{1}{n} - 2\frac{\rho}{\Zan} \frac1n + \frac{3M_2 + \tilde c_\alpha}{(\Delta_1(n)+\alpha\Delta_2(n))^2} = \frac1n \Bigg(\frac{\Zan - 2\rho}{\Zan} + (3M_2 + \tilde c_\alpha) Y(n) \Bigg),
\end{equation}
which gives
  \begin{align*}
    & Y(n) \leq \frac{2\rho-\Zan}{\Zan(3M_2 + \tilde c_\alpha)} 
      \implies{} \Exp*{Y(n+1) - Y(n) \given \cF(n)} \leq 0. \addtag \label{eq:Yn-condition}
  \end{align*}
  % A necessary condition for the first equation is that the righthandside must be positive
  % (since $Y(n)$ is a positive process). The denominator is always positive by
  % definition, whereas taking a closer look at the numerator requires considering the models
  % separately, since the values of $\rho$ and $\Zan$ differ.
%
For $n$ between $\taugood(A)$ and
  $\taubad(B)$,
  \[\Delta_1(n) + \alpha \Delta_2(n) > B\Zan \sqrt{n},\addtag\]
  implying
  \begin{equation}
    Y(n) \leq \frac{n}{\left( B\Zan \sqrt{n} \right)^2} = \frac{1}{\left(B\Zan \right)^2}.
  \end{equation}
Further,
  \begin{equation}
      2\rho-\Zan > 0 \iff q < \begin{cases}
          \frac{\alpha(1-\frac1n)+1}{4} & \text{ for very simple increasing trees} \\
          \frac{2\alpha(1-\frac1n)+1}{4(\alpha+1)} & \text{ for shape exchangeable trees,}
      \end{cases}
  \end{equation}
  and with \[B \geq \sqrt{\frac{3M_2 + \tilde c_\alpha}{\Zan(2\rho-\Zan)}} \quad \text{ for } q < \frac16, \addtag\] 
  the claim follows.
\end{proof}

\begin{proof}[Proof of~\cref{lem:2-bound}]
We now use this supermartingale to prove the lower bound for
  \[\Prob{\taugood(A) \leq N, \taubad(B) > N},\]
  where we must both first enter the regime $\Delta_1(n)+\alpha\Delta_2(n) > A\sqrt{n}$ and then not fall below $B\sqrt{n}$ again. Note
  \begin{equation*}\label{eq:circled2}
    \Prob*{\taugood(A) \leq N, \taubad(B) > N} = 1 - \Prob*{\taugood(A) > N} - \Prob*{\taubad(B) \leq N}. \addtag
  \end{equation*}

  The decreasing tendency of $Y(n)$ after $\taugood(A)$ given by~\cref{lem:Y-supmart} implies that $\Delta_1(n)+\alpha\Delta_2(n)$ has an increasing tendency after this point. In particular, $Y(n)$ compares $\Delta_1(n)+\alpha\Delta_2(n)$ to $\sqrt{n}$ which we will use to bound $\Prob*{\taubad(B) \leq N}$. 
  
  In the following, we omit the parameters $A$ and $B$ in the notation whereever they are not relevant. Remember that we consider $Y(n)$ as a process adapted to $\cF(n)$, the natural filtration of $\Delta(n)$. Then, via the definition of $\taugood$ and $\taubad$,
%  since these two events are disjoint by definition ($\taubad \geq \taugood$).
%
\begin{align*}
      Y(\taugood) \leq{}&
    \frac{\taugood}{(AZ_\alpha(\taugood)\sqrt{\taugood})^2} =
    \frac{1}{A^2Z_\alpha(\taugood)^2} \leq \frac{1}{A^2}, \addtag \\
    Y(\taubad) \geq{}&
    \frac{\taubad}{(BZ_{\alpha}(\taubad)\sqrt{\taubad})^2} =
    \frac{1}{B^2Z_\alpha(\taubad)^2}\\ \geq{}&
    \begin{cases}
      \frac{1}{B^2(\alpha+2)^2} & \text{ for very simple increasing trees} \\
      \frac{1}{B^2(2\alpha+2)^2} & \text{ for shape exchangeable trees}.
    \end{cases} \addtag
  \end{align*}
Additionally,
\begin{equation*}
    \Prob*{\taubad \leq N} = \Exp*{\1_{N \wedge \taubad = \taubad}} = \Exp[\big]{\Exp*{\1_{N \wedge \taubad =
          \taubad} \given \cF(\taugood)}}.
\end{equation*}
On the event $\taugood  \leq N$, we get % $\left(Y(n \wedge \taubad)\right)_{n \geq \taugood}$ is a supermartingale and 
by a variant of the optimal stopping theorem~\cite[Theorem 28, Chapter V]{meyer1966probability}
\begin{align*}
  & \Exp*{Y(N \wedge \taubad) \given \cF(\taugood)} \leq Y(\taugood). \addtag \\
  \intertext{Further, since $Y(\taugood \lor n \land \taubad)$ is always positive,}
    & \Exp*{Y(N \wedge \taubad) \given \cF(\taugood)} \\
    ={}& \Exp{Y(N)\1_{N \wedge \taubad = N} + Y(\taubad)\1_{N \wedge \taubad = \taubad} \given \cF(\taugood)} \\
    \geq{}& \Exp*{Y(\taubad)\1_{N \wedge \taubad = \taubad} \given \cF(\taugood)}. \addtag
\end{align*}

Therefore,
\begin{align*}
    & \frac{1}{A^2} \geq \Exp*{Y(\taugood)} \geq \Exp[\big]{\Exp*{\1_{N \wedge \taubad =
          \taubad} \given \cF(\taugood)}} \cdot \begin{cases} \frac{1}{B^2(\alpha+2)^2} & \text{ for very simple increasing trees} \\
              \frac{1}{B^2(2\alpha+2)^2} & \text{ for shape exchangeable trees}
            \end{cases} \\
    \iff{}& \frac{1}{A^2} \geq \Prob*{\taubad \leq N} \cdot \begin{cases} \frac{1}{B^2(\alpha+2)^2} & \text{ for very simple increasing trees} \\
              \frac{1}{B^2(2\alpha+2)^2} & \text{ for shape exchangeable. trees,}
            \end{cases} \addtag
\end{align*}
which by definition of $A$ implies
\begin{equation}\label{eq:taubadbound}
    \Prob*{\taubad \leq N} \leq q^{1-2\gamma} \cdot \begin{cases} B^2(\alpha+2)^2 & \text{ for very simple increasing trees} \\
              B^2(2\alpha+2)^2 & \text{ for shape exchangeable trees.}
            \end{cases}
\end{equation}
%  (Remember that $B$ does not depend on $q$ at all.)

  To bound $\Prob*{\taugood(A) \leq N}$, we consider the first point at which
  $\Delta_1(n) + \alpha \Delta_2(n)$ may reach this boundary. For it to happen at time $n_0$, $\Delta_1(n_0) = n_0$, $\Delta_2(n_0) = n_0-1$ must hold, implying $\Delta_1(n_0) + \alpha \Delta_2(n_0) = \Zan n_0$. For both model groups,
  \[Z_\alpha(n_0) n_0 > AZ_\alpha(n_0) \sqrt{n_0} \iff \sqrt{n_0} > A, \addtag \]
  which is fulfilled for $n_0 > A^2$.
  Note that, since
  $\Delta_1(1) = 1$ and, by definition of $D(n)$, $\Delta_2(2) = 1$ a.s., it
  holds that
  \begin{align*}
      & \Prob[\big]{\Delta_1(n_0) + \alpha \Delta_2(n_0) = Z_\alpha(n_0) n_0} = (1-q)^{n_0-1}. \addtag \\
      \intertext{Therefore, the probability that the
  process will not reach the $\taugood(A)$-boundary before time $N$ is bounded from above by}
   & \Prob*{\taugood(A) > N} \leq 1 - (1-q)^{\ceil{A^2}-1} \leq
    (A^2+1)q - q = q^{2\gamma}. \addtag \label{eq:taugoodbound}
  \end{align*}
%  and
% 
%
Plugging~\cref{eq:taubadbound,eq:taugoodbound} into~\cref{eq:circled2} gives
  \begin{align*}
    \Prob{\taugood(A) \leq N, \taubad(B) > N} \geq{}& \begin{cases} 1 - q^{2\gamma} - B^2(\alpha+2)^2q^{1-2\gamma} & \text{ for very simple increasing trees} \\% \conv{N \to \infty} 1 - c_{\alpha,2}\sqrt{q} \addtag \\  %\geq{} 1 - 2B^2(\alpha+2)^2\sqrt{q} + q \geq 1 - c_{\alpha,2}\sqrt{q} \addtag \\ 
     1 - q^{2\gamma} - B^2(2\alpha+2)^2q^{1-2\gamma} & \text{ for shape exchangeable trees.} 
     \end{cases} \addtag %\to_{N \to \infty} 1 - c_{\alpha,2}\sqrt{q}. \addtag %\geq{} 1 - 2B^2(2\alpha+2)^2\sqrt{q} + q \geq 1 - c_{\alpha,2}\sqrt{q}. \addtag
  \end{align*}
  With~\cref{rem:Bconv} and maximizing the above expression at $\gamma = \frac14$,
  \[\liminf_{N \to \infty}{}\: \Prob{\taugood(A) \leq N, \taubad(B) > N} \geq 1 - c_{\alpha,2}\sqrt{q} \addtag\]
  follows.
\end{proof}

% Local Variables:
% mode: latex
% TeX-engine: luatex
% TeX-master: "../main.tex"
% End:
\subsection{Proving~\texorpdfstring{\cref{lem:1-bound}}{Lemma 12}}\label{sec:lem-cheby}
\newcommand{\cC}{\mathcal{C}}
Recall
\begin{equation}
    \Zan \coloneqq
    \begin{cases}
      \alpha(1-\frac1n)+1 & \text{ for very simple increasing trees} \\
      2\alpha(1-\frac1n)+1 & \text{ for shape exchangeable trees.}
    \end{cases}
\end{equation}

\begin{proof}[Proof of~\cref{lem:1-bound} for $\alpha = 0$]
  If $\alpha = 0$, $Z_{\alpha=0}(n) \equiv 1$ and it holds that \begin{align*}
      \taugood(A) ={}& \inf \Set*{n > 0 \given \Delta_1(n) > A\sqrt{n}} \\
      \taubad(B) ={}& \inf \Set*{n > \taugood \given \Delta_1(n) \leq B \sqrt{n}}
  \end{align*}
  and therefore 
  \[\{\taugood \leq N, \taubad > N\} \subset \{\Delta_1(N) >
  0\}, \addtag \label{eq:unifcase} \] 
  which implies $\Prob*{\Delta_1(N) > 0 \given \taugood \leq N, \taubad > N} = 1 \geq 1 - c_{1,\alpha = 0}\sqrt{q}$ for any $c_{1,\alpha=0} > 0$.
\end{proof}

For $\alpha \neq 0$,~\cref{eq:unifcase} does not hold. The core idea of the following argument is that with high enough
probability, $\Delta_1(n)$ and $\Delta_2(n)$ will not stray too far from each other.

\begin{definition}\label{def:additionalRVs}
  Let $[N] \coloneqq \{1, 2, \ldots, N\}$.
  We define the following random variables representing the decisions made by $\Delta(n)$ up to time $N$.
  % all measurable by the product measure induced by $(\Delta(n))_{n \leq N}$:
  \begin{itemize}
  \item $\rr([N]), \rb([N]), \br([N]), \bb([N])$:\\
  The first entry in the tuple represents the color of the
    attached-to vertex and the second entry the color of the attaching vertex. Each random variable takes on values in $\mathcal{P}(\Set*{2, \ldots,
      N})$ such that, e.g., $k \in (b,r)([N])$ iff a new red vertex
    attached to an existing blue vertex at the transition from time $k-1$ to
    $k$.
  \item $\rw([N]), \wr([N]), \bw([N]), \wb([N])$:\\
    These random variables are defined as (disjoint) unions of the random variables defined above. For example, $\rw([N]) = \rr([N]) \cup \rb([N])$.
  \end{itemize}
\end{definition}

  Additionally, we define the event
  \[\cA \coloneqq \Set*{\srwN \in \left[\frac{\swrN - Nq - a\sqrt{N}}{1-2q},
        \frac{\swrN - Nq + a\sqrt{N}}{1-2q}\right]}, \addtag \] 
  where $a = \frac{B \ZaN}{4\abs*{\alpha}} > 0$ and we write $\#_S$ to denote the cardinality of a set $S$. Finally, we introduce the shorthand notation $E_\tau$ for the event  $\taugood(A) \leq N, \taubad(B) > N$. 

With these definitions in place, we are ready to prove~\cref{lem:1-bound} for
$\alpha \neq 0$.

\begin{proof}[Proof of~\cref{lem:1-bound} for $\alpha \neq 0$]
We show
\[\liminf_{N \to \infty} \Prob{\cA \given E_\tau} \geq 1 - c_{\alpha,1}\sqrt{q} \quad \text{ and } \quad \Prob{\Delta_1(N) > 0 \given \cA\cap E_\tau} = 1 \addtag,\]
which implies the claim via
\[\Prob{\Delta_1(N) > 0 \given \taugood \leq N, \taubad > N} = \Prob{\Delta_1(N) > 0 \given \cC} \geq \Prob{\Delta_1(N) > 0 \given \cA\cap E_\tau}\Prob{\cA \given E_\tau}. \addtag\]
Note that it is again sufficient to prove both these claims for small values of $q$. It holds that
  \begin{align*}\label{eq:delta1-rewrite}
    \Delta_1(N) ={}& \swrN - \swbN \\
    ={}& \swrN  - (N - \swrN) \\
    ={}& 2\swrN - N \addtag
  \end{align*}
  for both tree models. Further, in the very simple increasing tree case,
  \begin{align*}\label{eq:delta2-vs-rewrite}
    \Delta_2(N) ={}& \srwN - \sbwN \\
    ={}& \srwN  - (N - \srwN) \\
    ={}& 2\srwN - N \addtag
  \end{align*}
  and for shape exchangeable trees,
  \begin{align*}\label{eq:delta2-se-rewrite}
    \Delta_2(N) ={}& 2\srrN + 0\srbN + 0\sbrN - 2\sbbN \\
    ={}& 2\srrN - 2(N - \srbN - \sbrN - \srrN) \\
    %={}& 2\srrN - 2N + 2\srbN + 2\sbrN + 2\srrN \\
    %={}& 4\srrN + 2\srbN + 2\sbrN - 2N \\
    %={}& 2\srrN + 2\srbN + 2\srrN + 2\sbrN - 2N \\
    ={}& 2\srwN + 2\swrN - 2N. \addtag
  \end{align*}
To lower-bound  $\Prob{\cA \given E_\tau}$, note that $\Prob{\cA \given E_\tau} \geq \Prob{\cA} - \Prob*{E_\tau^c}$ %with~\cref{lem:2-bound},
and that given the color of the drawn vertex, the color of the new vertex is simply an independent coin flip. Let $\mathbb{P}_i$ be the probability measure conditioned on $\srwN = i$: % Therefore, we have
  \begin{align}
  & \E_i \left[\swrN \right] = (1-q)i + q(N-i) = (1-2q)i + Nq \\
  \intertext{and}
  &\Vari_i \left[\swrN \right] = (1-q)qi + (1-q)q(N-i) \leq qi +
    q(N-i) = qN.
  \end{align}
  By definition of $\cA$,
  \begin{align*}
     \mathbb{P}_i \left(\cA^c \right) ={}& \mathbb{P}_i \left(\abs*{\swrN - \E_i \left[\swrN\right]} > a \sqrt{N}\right) \\
     \leq{}& \frac{qN}{a^2N} = \frac{1}{a^2}q \addtag
  \end{align*}
  and thereby
  \begin{equation}
    \mathbb{P}_i \left(\cA \right) \geq 1 - \left(\frac{4\alpha}{B\ZaN}\right)^2 q \geq 1 - \hat c_\alpha q.
  \end{equation}
  uniformly in $i$, bringing us together with~\cref{rem:Bconv} and~\cref{lem:2-bound} to
    \begin{equation}
    \liminf_{N \to \infty} \Prob{\cA \given E_\tau} \geq 1 - (\hat c_\alpha+c_{\alpha,2})\sqrt{q} = 1 - c_{\alpha,1}\sqrt{q}.  \label{eq:first-cond}
  \end{equation}
  It remains to prove
  \[\Prob*{\Delta_1(N) > 0 \given \cA\cap E_\tau} = 1. \addtag \label{eq:second-cond}\]
Note that $\{\Delta_1(N) > 0\} = \{\swrN > \frac{N}{2}\}$ and let $\omega \in \cA \cap E_\tau$. By~\cref{def:additionalRVs} we have
  \[\frac{\swrN(\omega) - Nq - a \sqrt{N}}{1-2q} \leq \srwN(\omega) \leq \frac{\swrN(\omega) - Nq + a\sqrt{N}}{1-2q}
    \addtag \label{eq:interval}\]
  and
  \[\Delta_1(N)(\omega) + \alpha \Delta_2(N)(\omega) > B\ZaN\sqrt{N} \addtag \label{eq:totalsize}\]
  (we now omit $(\omega)$ for the sake of readability).
%  It remains to show that these events together imply $\srwN > \frac{N}{2}$, which is equivalent to $\Delta_1(N) > 0$.
\cref{eq:delta1-rewrite,eq:totalsize} together with~\cref{eq:delta2-vs-rewrite} give for very simple increasing trees
  \begin{align*}
    0 <{}& \Delta_1(N) + \alpha \Delta_2(N) - B\ZaN\sqrt{N} \\
    ={}& 2(\swrN + \alpha \srwN) - (1+\alpha)N - B\ZaN\sqrt{N}. \addtag \label{eq:vinc-1}
    \intertext{Replace~\cref{eq:delta2-vs-rewrite} with~\cref{eq:delta2-se-rewrite} to get}
    0 <{}& \Delta_1(N) + \alpha \Delta_2(N) - B\ZaN\sqrt{N} \\
    ={}& (2+2\alpha)\swrN + 2\alpha \srwN - (1+2\alpha)N - B \ZaN \sqrt{N} \addtag \label{eq:shex-1}
  \end{align*}
  for shape exchangeable trees. If $\alpha > 0$, we apply the upper interval bound from~\cref{eq:interval},
  which together with~\cref{eq:vinc-1} gives for very simple increasing trees:
  \begin{align*}
    0 <{}& 2(\swrN + \alpha \srwN) - (1+\alpha)N - B\ZaN\sqrt{N} \\
    \leq{}& 2\left(\swrN + \alpha \frac{\swrN - Nq + a \sqrt{N}}{1-2q}\right) - (1+\alpha)N - B\ZaN\sqrt{N} \\
    ={}& 2\swrN\left( 1 + \alpha \frac{1}{1-2q} \right) - \left(1+\alpha+ \frac{2\alpha q}{1-2q}\right)N - \left( B\ZaN - \frac{2 \alpha a}{1-2q}\right)\sqrt{N}\\
    ={}& 2\swrN\left( 1 + \frac{\alpha}{1-2q} \right) - \left(1+\frac{\alpha}{1-2q}\right)N - \left( B\ZaN - \frac{2 \alpha a}{1-2q}\right)\sqrt{N}.
    \addtag \label{eq:vinc-2}
    \intertext{and for shape exchangeable trees, we continue from~\cref{eq:shex-1}:}
    0 <{}& (2+2\alpha)\swrN + 2\alpha \srwN - (1+2\alpha)N - B \ZaN \sqrt{N} \\
    \leq{}& (2+2\alpha)\swrN + 2\alpha\left( \frac{\swrN - Nq +
            a\sqrt{N}}{1-2q} \right)  - (1+2\alpha)N - B \ZaN \sqrt{N} \\
    %={}& (2+2\alpha+\frac{2\alpha}{1-2q})\swrN - (1+2\alpha+\frac{2\alpha q}{1-2q})N - \left(B \ZaN - \frac{2\alpha a}{1-2q}\right) \sqrt{N} \\
    ={}& 2\swrN\left(1+\alpha+\frac{\alpha}{1-2q}\right) - \left(1+2\alpha+\frac{2\alpha q}{1-2q}\right)N - \left(B \ZaN - \frac{2\alpha a}{1-2q}\right) \sqrt{N}\\
    ={}& 2\swrN\left(1+\frac{2\alpha-2\alpha q}{1-2q}\right) - \left(1+\frac{2\alpha-2\alpha q}{1-2q}\right)N - \left(B \ZaN - \frac{2\alpha a}{1-2q}\right) \sqrt{N}.
    \addtag \label{eq:shex-2}
  \end{align*}
It holds that
\begin{align*}
     a = \frac{B \ZaN}{4\alpha} \overset{q < \frac12}{\implies{}}& a < (1-2q)\frac{B\ZaN}{2\alpha} \\
    \overset{\alpha > 0}{\iff{}}& B\ZaN >2\alpha \frac{a}{1-2q}. \addtag
\end{align*}
  In the very simple increasing tree case we continue from~\cref{eq:vinc-2}
  \begin{align*}
    % 0 <{}& 2(\swrN+ \alpha \srwN) - (1+\alpha)N - B\ZaN\sqrt{N} \\
    0 <{}& 2\swrN\left( 1 + \frac{\alpha}{1-2q} \right) - \left(1+\frac{\alpha}{1-2q}\right)N - \left( B\ZaN - \frac{2 \alpha a}{1-2q}\right)\sqrt{N} \\
    <{}& \left( 1 + \frac{\alpha}{1-2q} \right)(2\swrN - N)\addtag \label{eq:rebecca1}
    %<{}& 2\swrN\left( 1 + \alpha \frac{1}{1-2q} \right) - \left(1+\alpha+ \frac{2\alpha q}{1-2q}\right)N \\
    %={}& \left( 1 + \frac{\alpha}{1-2q} \right)(2\swrN - N) 
    \intertext{and in shape exchangeable trees from~\cref{eq:shex-2}}
    0 <{}& 2\swrN\left(1+\frac{2\alpha-2\alpha q}{1-2q}\right) - \left(1+\frac{2\alpha-2\alpha q}{1-2q}\right)N - \left(B \ZaN - \frac{2\alpha a}{1-2q}\right) \sqrt{N}\\
    <{}& \left(1 + \frac{2 \alpha - 2\alpha q}{1-2q}\right)\left( 2\swrN - N\right).\addtag \label{eq:rebecca2}
    %0 < {}& \left(1+\alpha+\frac{\alpha}{1-2q}\right)\swrN - \left(1+2\alpha+\frac{2\alpha q}{1-2q}\right)N \\
    %={}& 2(1+\alpha+\frac{\alpha}{1-2q})\swrN - (1+2\alpha+\frac{2\alpha q}{1-2q})N \\
    %={}& \left(1 + \frac{2 \alpha - 2\alpha q}{1-2q}\right)\left( 2\swrN - N\right). \addtag \label{eq:rebecca2}
  \end{align*}
  Both~\cref{eq:rebecca1,eq:rebecca2} imply $\swrN > \frac{N}{2}$ for $q < \frac12$, proving~\cref{lem:1-bound} for $\alpha > 0$.

  For negative values of $\alpha$, we use the lower interval bound from~\cref{eq:interval}. This only changes the $\sqrt{N}$-term, giving us for very simple increasing trees
  \begin{align*}
    0 <{}& 2(\swrN + \alpha \srwN) - (1+\alpha)N - B\ZaN\sqrt{N} \\
    \leq{}& 2\swrN\left( 1 + \frac{\alpha}{1-2q} \right) - \left(1+\frac{\alpha}{1-2q}\right)N - \left( B \ZaN + \frac{2\alpha a}{1-2q}\right)\sqrt{N} \addtag \label{eq:vinc-3}
    \intertext{and for shape exchangeable trees}
    0 <{}& (2+2\alpha)\swrN + 2\alpha \srwN - (1+2\alpha)N - B \ZaN \sqrt{N} \\
    \leq{}& 2\swrN\left(1+\frac{2\alpha-2\alpha q}{1-2q}\right) - \left(1+\frac{2\alpha-2\alpha q}{1-2q}\right)N - \left(B \ZaN + \frac{2\alpha a}{1-2q}\right) \sqrt{N}. \addtag \label{eq:shex-3}
  \end{align*}
We see that the first two terms correspond to~\cref{eq:vinc-2,eq:shex-2}. Again, the $\sqrt{N}$-term is positive by our choice of $a$, while the first two terms remain positive for $q < \frac{\alpha + 1}{2}$ for very simple increasing trees and $q < \frac{1 + 2\alpha}{2(1+\alpha)}$ for shape exchangeable trees and $\alpha$ in the respective allowed ranges. This finishes the proof of~\cref{lem:1-bound}.

\end{proof}

% Local Variables:
% mode: latex
% TeX-engine: luatex
% TeX-master: "../main.tex"
% End:

\section*{Acknowledgements}%%%%%%%%%%

%\footnote{% Place the text of your acknowledgements after the \acks (or \Acks) command. This will generate the heading "Acknowledgements". If you wish to make only one acknowledgement, use \ack (or \Ack).
We thank Stephan Wagner for useful discussions on the broadcasting problem and on \cite{DW19}. RS thanks C\'ecile Mailler for useful discussions on Elephant Random Walks. We also thank the anonymous reviewers for their helpful comments and for bringing the possible extension to $\alpha = -\frac12$ for very simple increasing trees to our attention.

This research is supported by the internal research funding (Stufe I) at Johannes Gutenberg-University, the TOP-ML project and the Deutsche Forschungsgemeinschaft (DFG, German Research Foundation) through Project-ID 233630050 - TRR 146.

% There were no competing interests to declare which arose during the preparation or publication process of this article.

\printbibliography

\end{document}